\RequirePackage{etex}
\RequirePackage{easymat}

\documentclass[12pt]{elsarticle}
\usepackage{amsthm,amsmath,amssymb,
MnSymbol,tikz}
\usepackage[all]{xy}

\usepackage[active]{srcltx}
\renewcommand{\le}{\leqslant}
\renewcommand{\ge}{\geqslant}
\renewcommand{\dashrightarrow}
{\text{\raisebox{0.9mm} {\
\begin{tikzpicture}[->,thick,xscale=0.56]
  \draw[dashed] (0,0)--(1,0)
;
\end{tikzpicture}}\ }}

\renewcommand{\dashleftarrow}
{\text{\raisebox{0.9mm} {\
\begin{tikzpicture}[<-,thick,xscale=0.56]
  \draw[dashed] (0,0)--(1,0)
;
\end{tikzpicture}}\ }}

\newcommand{\lin}{\,\frac{}{\quad}\,}
\newcommand{\li}[1]{\ \frac{#1}
{\qquad}\ }

\newtheorem{theorem}{Theorem}
\newtheorem{lemma}[theorem]{Lemma}

\theoremstyle{definition}
\newtheorem{definition}[theorem]{Definition}

\begin{document}

\title{Cycles of linear and semilinear mappings}

\author[br]{Debora Duarte de Oliveira}
\ead{d.duarte.oliveira@gmail.com}

\author[br]{Vyacheslav Futorny}
\ead{futorny@ime.usp.br}

\author[klim]{Tatiana Klimchuk}
\ead{klimchuk.tanya@gmail.com}

\author[klim]{Dmitry Kovalenko}

\author[serg]{Vladimir V. Sergeichuk\corref{cor}}
\ead{sergeich@imath.kiev.ua}

\address[br]{Department of Mathematics, University of S\~ao Paulo, Brazil.}

\address[klim]
{Faculty of Mechanics and Mathematics,
Kiev National Taras Shevchenko
University, Kiev, Ukraine.}

\address[serg]
{Institute of Mathematics,
Tereshchenkivska 3, Kiev, Ukraine.}

\cortext[cor]{Corresponding author.}

\begin{abstract}
We give a canonical form of matrices of
a cycle of linear or semilinear mapping
$V_1\,\text{---}\,V_2\,
\text{---}\:\cdots\:\text{---}\,V_t\,
\text{---}\,V_1$ in which all $V_i$ are
complex vector spaces, each line is an
arrow $\longrightarrow$ or
$\longleftarrow$, and each arrow
denotes a linear or semilinear mapping.
\end{abstract}


\maketitle

\section{Introduction}

A mapping ${\cal A}$ from a complex
vector space $U$ to a complex vector
space $V$ is \emph{semilinear} if
\begin{equation*}\label{lkhe}
{\cal A}(u+u')={\cal A}u+{\cal A}u',\qquad
{\cal A}(\alpha u)=\bar\alpha {\cal A}u
\end{equation*}
for all $u,u'\in U$ and $\alpha
\in\mathbb C$. We write ${\cal A}: U\to
V$ if ${\cal A}$ is a linear mapping
and ${\cal A}: U\dashrightarrow V$
(using a dashed arrow) if ${\cal A}$ is
a semilinear mapping.

We give a canonical form of matrices of
a \emph{cycle} of linear and semilinear
mappings
\begin{equation}\label{jsttw}
\begin{split}
{\cal A}: \qquad\xymatrix{
{V_1}
\ar@{-}@/^2pc/[rrrr]^{\mathcal A_t}
\ar@{-}[r]^{\ \ \mathcal A_1}&
V_2\ar@{-}[r]^{\mathcal A_2} &
{\ \dots\ }&{V_{t-1}}
\ar@{-}[l]_{\mathcal A_{t-2}}
\ar@{-}[r]^{\ \mathcal A_{t-1}\ \ }&{V_t}}
\end{split}
\end{equation}
in which each line is a full arrow
$\longrightarrow$ or $\longleftarrow$,
or a dashed arrow $\dashrightarrow$ or
$\dashleftarrow$.

Its special cases are the canonical
forms of
\begin{itemize}
  \item matrix pencils,
      contragredient matrix
      pencils, and pairs consisting
      of a linear
mapping and a semilinear mapping
      (i.e., cycles
      $V_1\begin{matrix}
      \longrightarrow \\[-3.5mm]
      \longrightarrow
    \end{matrix}V_2$,
$V_1\begin{matrix}
      \longrightarrow \\[-3.5mm]
      \longleftarrow
    \end{matrix}V_2$,
and $V_1\begin{matrix}
      \longrightarrow \\[-3.5mm]
      \dashrightarrow
    \end{matrix}V_2$); their
canonical matrices were obtained by
Kronecker \cite{kro},
Dobrovol$'$skaya and Ponomarev
\cite{dob+pon} (see also
      \cite{hor+mer}), and
Djokovi\'{c} \cite{djok};

  \item cycles of linear mappings
      (all arrows in \eqref{jsttw}
      are full); their canonical
      form is well known in the
      theory of quiver
      representations; see
      \cite[Section 11.1]{gab_roi}
      and \cite{ser_cycl}.
\end{itemize}

The proof is based on the canonical
forms of linear and semilinear
operators (i.e., cycles
\!\!\raisebox{-2mm}
      {\begin{tikzpicture}[->,thick,xscale=1.3,
      yscale=1.3]
  \node (1) {$V_1$\!\!}; \path (1)
    edge [loop right]
(1);
\end{tikzpicture}}\!\!
and \!\!\raisebox{-2mm}
      {\begin{tikzpicture}[->,thick,xscale=1.3,
      yscale=1.3]
  \node (1) {$V_1$\!\!}; \path
    (1) edge [dashed,loop right]
(1);
\end{tikzpicture}}\!\!) given
by Jordan \cite[p. 125]{jordan} and
Haantjes \cite{haa} (see also Theorem
\ref{tej}). We use the methods of
\cite{n-r-s-b}.

In Section  \ref{s2} we formulate
Theorem \ref{t1.1} about a canonical
form of matrices of \eqref{jsttw}; in
Section \ref{s3} we prove it as
follows. In Section \ref{ss1} we reduce
the proof of Theorem \ref{t1.1} to
cycles \eqref{jsttw} in which $t\ge 2$
and the mappings $\mathcal A_1,\dots,
\mathcal A_{t-1}$ are linear. In
Section \ref{ss2} we give a canonical
form of the matrices of $\mathcal
A_1,\dots, \mathcal A_{t-1}$. In
Section \ref{ss3} we give a canonical
form of the matrix of $\mathcal A_t$
with respect to those transformations
that preserve the matrices of $\mathcal
A_1,\dots, \mathcal A_{t-1}$. In
Section \ref{ss4} we complete the proof
of Theorem \ref{t1.1}.

All matrices and vector spaces that we
consider are over the field of complex
numbers.

\section{A canonical form of matrices of
a cycle} \label{s2}

We denote by $[v]_e$ the coordinate
vector of $v$ in a basis
$e_1,\dots,e_n$, and by $S_{e\to e'}$
the transition matrix from
$e_1,\dots,e_n$ to a basis
$e'_1,\dots,e'_n$.  We write $\bar
A:=[\bar a_{ij}]$ for a matrix
$A=[a_{ij}]$.

Let ${\cal A}: U\dashrightarrow V$  be
a semilinear mapping. We say that an
$m\times n$ matrix ${\cal A}_{fe}$ is
the \emph{matrix of $\cal A$ in bases}
$e_1,\dots,e_n$ of $U$ and
$f_1,\dots,f_m$ of $V$ if
\begin{equation}\label{feo2}
[{\cal A}u]_f=\overline{{\cal A}_{fe}[u]_e}\qquad
\text{for all }u\in U.
\end{equation}
Therefore, the columns of ${\cal
A}_{fe}$ are $\overline{[{\cal
A}e_1]_f}, \dots, \overline{[{\cal
A}e_n]_f}$. We write ${\cal A}_{e}$
instead of ${\cal A}_{ee}$ if $U=V$.

The \emph{direct sum} of matrix
sequences $A=(A_1,\dots,A_t)$ and
$B=(B_1,\dots,B_t)$ is the sequence
\[
A\oplus B:=\left(
\begin{bmatrix}
  A_1 & 0 \\
  0 & B_1 \\
\end{bmatrix},\dots, \begin{bmatrix}
  A_t & 0 \\
  0 & B_t \\
\end{bmatrix}
\right).
\]

For each $k = 1,2,\dots$, define the
matrices
\begin{equation*}\label{1aaa}
J_k(\lambda):=\begin{bmatrix}
\lambda&1&&0\\&\lambda&\ddots&\\&&\ddots&1
\\ 0&&&\lambda
\end{bmatrix}\qquad (\text{$k$-by-$k$, }\lambda
\in\mathbb C),
\end{equation*}
\begin{equation*}\label{kebw}
H_{2k}(\mu):=\begin{bmatrix}
              0 & I_k \\
              J_{k}(\mu) & 0 \\
            \end{bmatrix}\qquad
(\text{$2k$-by-$2k$, }\mu
\in\mathbb C),
\end{equation*}
and
\begin{equation}\label{1.4}
F_n=\begin{bmatrix}
0&1&&0\\&\ddots&\ddots&\\0&&0&1
\end{bmatrix},\quad
G_n=\begin{bmatrix}
1&0&&0\\&\ddots&\ddots&\\0&&1&0
\end{bmatrix}\qquad
(\text{$(n-1)$-by-$n$}).
\end{equation}

The following theorem is the main
result of the article.

\begin{theorem}
\label{t1.1} For each system of linear
and semilinear mappings \eqref{jsttw},
there exist bases of the spaces
$V_1,\dots,V_t$ in which the sequence
of matrices of
${\cal{A}}_{1},\dots,{\cal{A}}_{t}$ is
a direct sum, determined by
\eqref{jsttw} uniquely up to
permutation of summands, of sequences
of the following form $($in which the
points denote sequences of identity
matrices or $0_{00})$:
\begin{itemize}

\item[\rm{(i)}]
    \begin{itemize}

\item[$\bullet$]
    $(J_n(\lambda),\ldots)$ in
    which $\lambda\ne 0$, if
    the number of dashed arrows
    in \eqref{jsttw} is even,

\item[$\bullet$]
    $(J_n(\lambda),\ldots)$ and
    $(H_{2k}(\mu),\ldots)$ in
    which $\lambda $ is real
      and positive and $\mu $
      is either not real or is
      real and negative, if the
    number of dashed arrows in
    \eqref{jsttw} is odd;
\end{itemize}

\item[\rm{(ii)}] $(\ldots, J_n(0),
    \ldots)$ with $J_n(0)$ at
    position $i\in \{1,\ldots,
    t\}$;

\item[\rm{(iii)}] $(\ldots,
    A_i,\ldots, A_j,\ldots)$ in
    which
\begin{itemize}
  \item[$\bullet$]
      $(A_i,A_j)=(F_n,G_n)$ or
      $(F_n^T,G_n^T)$ if ${\cal
      A}_i$ and ${\cal
A}_j$ have opposite directions
in
\begin{equation}\label{1.6}
  V_1\li{{\cal A}_1}
  V_2\li{{\cal A}_2}  \cdots
  \li{{\cal A}_{t-1}}
  V_t\li{{\cal A}_t} V_1
\end{equation}

  \item[$\bullet$] $(A_i,A_j)=
      (F_n,G_n^T)$ or
      $(F_n^T,G_n)$ if ${\cal
      A}_i$ and ${\cal A}_j$
have the same direction in
\eqref{1.6}.

\end{itemize}
\end{itemize}
\end{theorem}

Note that $F_1$ and $G_1$ in
\eqref{1.4} have the size $0\times 1$.
It is agreed that there exists exactly
one matrix, denoted by $0_{n0}$, of
size $n\times 0$ and there exists
exactly one matrix, denoted by
$0_{0n}$, of size $0\times n$ for every
nonnegative integer $n$; they represent
the linear mappings $0\to {\mathbb
C}^n$ and ${\mathbb C}^n\to 0$ and are
considered as zero matrices. Then
$$
M_{pq}\oplus 0_{m0}=\begin{bmatrix}
  M_{pq} & 0 \\
  0 &0_{m0}
\end{bmatrix}=\begin{bmatrix}
  M_{pq}& 0_{p0} \\
  0_{mq}& 0_{m0}
\end{bmatrix}=\begin{bmatrix}
M_{pq} \\ 0_{mq}
\end{bmatrix}
$$
and
$$
M_{pq}\oplus 0_{0n}=\begin{bmatrix}
  M_{pq} & 0 \\
  0 & 0_{0n}
\end{bmatrix}=\begin{bmatrix}
  M_{pq}& 0_{pn} \\
  0_{0q}& 0_{0n}
\end{bmatrix}=\begin{bmatrix}
   M_{pq} & 0_{pn}
\end{bmatrix}
$$
for every $p\times q$ matrix $M_{pq}$.

\section{Proof of Theorem \ref{t1.1}}
\label{s3}

\subsection{Reduction to cycles in which $t\ge 2$
and $\mathcal A_1,\dots,\mathcal
A_{t-1}$ are linear}\label{ss1}

The matrix of the composition of a
linear mapping and a semilinear mapping
is given in the following lemma.

\begin{lemma}\label{kowq}
Let $e_1,e_2,\dots,$ be a basis of a
vector space $U$, $f_1,f_2,\dots,$ be a
basis of $V$, and $g_1,g_2,\dots$ be a
basis of  $W$.
\begin{itemize}
  \item[\rm (a)] The composition of
      a linear mapping ${\mathcal
      A}: U\to V$ and a semilinear
      mapping ${\mathcal B}:
      V\dashrightarrow W$ is the
      semilinear mapping with
      matrix
\begin{equation}\label{lif}
(\mathcal B \mathcal
      A)_{ge}=\mathcal B_{gf}
      \mathcal A_{fe}
\end{equation}

 \item[\rm (b)] The composition of
     a semilinear mapping
     ${\mathcal A}:
     U\dashrightarrow V$ and a
     linear mapping ${\mathcal B}:
     V\to W$ is the semilinear
     mapping with matrix
\begin{equation}\label{kif}
(\mathcal B \mathcal
A)_{ge}=\overline{\mathcal B}_{gf} \mathcal
A_{fe}
\end{equation}

\end{itemize}
\end{lemma}

\begin{proof}
The identity \eqref{lif} follows from
observing that ${\cal AB}$ is a
semilinear mapping and
\[
[({\cal B}{\cal A})u]_{g}=
[{\cal B}({\cal A}u)]_{g}=
\overline{{\cal B}_{gf}[{\cal
A}u]_{f}}= \overline{({\cal
B}_{gf}{\cal A}_{fe})[u]_{e}}
\]
for each $u\in U$. The identity
\eqref{kif} follows from observing that
${\cal AB}$ is a semilinear mapping and
\[
[({\cal B}{\cal A})u]_{g}=
[{\cal B}({\cal A}u)]_{g}=
{\cal B}_{gf}[{\cal
A}u]_{f}= {\cal
B}_{gf}\overline{{\cal A}_{fe}[u]_{e}}=
\overline{(\overline{{\cal
B}}_{gf}{\cal A}_{fe})[u]_{e}}
\]
for each $u\in U$.
\end{proof}

Let $\mathcal A: U\dashrightarrow V$ be
a semilinear operator, let ${\cal
A}_{fe}$  be its matrix in bases
$e_1,\dots,e_m$ of $U$ and
$f_1,\dots,f_n$ of $V$, and let ${\cal
A}_{f'e'}$ be its matrix in other bases
$e'_1,\dots,e'_m$ and
$f'_1,\dots,f'_n$. Then
\begin{equation*}\label{swk3}
{\cal A}_{f'e'}=\bar S_{f\to f'}^{-1}{\cal A}_{fe}
S_{e\to e'}
\end{equation*}
since the right hand matrix satisfies
\eqref{feo2} with $e',f'$ instead of
$e,f$:
\[
\overline{\bar S_{f\to f'}^{-1}{\cal A}_{fe}
S_{e\to e'}[v]_{e'}}=
S_{f\to f'}^{-1}\overline{{\cal A}_{fe}
[v]_{e}}=S_{f\to f'}^{-1}[{\cal A}
v]_f=[{\cal A}
v]_{f'}
\]
In particular, if $U=V$, then
\begin{equation*}\label{swk4}
{\cal A}_{e'}=\bar S_{e\to e'}^{-1}{\cal A}_{e}
S_{e\to e'}
\end{equation*}
and so ${\cal A}_{e'}$ and ${\cal
A}_{e}$ are consimilar; recall that two
matrices $A$ and $B$ are
\emph{consimilar} if there exists a
nonsingular matrix $S$ such that $\bar
S^{-1}AS=B$.

The following canonical form of a
matrix under consimilarity was obtained
in \cite[Theorem 3.1]{hon-hor}; see
also \cite[Theorem 4.6.12] {HJ12}.

\begin{theorem}[{\cite{hon-hor,HJ12}}]
\label{tej} Each square complex
matrix is consimilar to a direct sum,
uniquely determined up to permutation
of direct summands, of matrices of the
following three types:
\begin{description}
  \item[Type 0:] $J_k(0),\ k =
      1,2,\dots$;
  \item[Type I:] $J_k(\lambda  ),\
      k = 1,2,\dots,$ in which
      $\lambda $ is real and
      positive;
  \item[Type II:] $H_{2k}(\mu)$, $k
      = 1,2,\dots,$ in which $\mu $
      is either not real or is real
      and negative.
\end{description}
\end{theorem}

It suffices to prove Theorem \ref{t1.1}
for cycles $\cal A$ of the form
\eqref{jsttw} in which
\begin{equation}\label{dsr}
    t\ge 2\text{ and the mappings $\mathcal
      A_1,\dots, \mathcal A_{t-1}$
      are linear}
\end{equation}
since Theorem \ref{t1.1} for
      the cycles
      \!\!\raisebox{-2mm}
      {\begin{tikzpicture}[->,thick,xscale=1.3,
      yscale=1.3]
  \node (1) {$V_1$\!\!}; \path
    (1) edge [dashed,loop right]
(1);
\end{tikzpicture}}\!\! and  \!\!\raisebox{-2mm}
      {\begin{tikzpicture}[->,thick,xscale=1.3,
      yscale=1.3]
  \node (1) {$V_1$\!\!}; \path (1)
    edge [loop right]
(1);
\end{tikzpicture}}\!\!
of length $t=1$
  follows from
the Jordan canonical form and Theorem
\ref{tej}, and all $\mathcal
      A_1,\dots, \mathcal A_{t-1}$
can be made linear mappings by using
      the following procedure:

\begin{equation}\label{4KR}
\parbox{25em}
{Let not all $\mathcal
      A_1,\dots, \mathcal A_{t-1}$
      be linear and let $\mathcal
      A_{k-1}$ for some $k\le t$ be
      the first semilinear mapping.
      Denote by $\mathcal A^{(k)}$
      the cycle obtained from
      $\mathcal A$ by replacing the
      vector space $V_k$ by the
      vector space
$\bar V_k$ defined as follows:
$\bar V_k$
      consists of the same elements
      as $V_k$, the addition in
$\bar V_k$ is the same as in $V_k$,
and the multiplication (which we
denote by ``$\circ$'') in $\bar
V_k$ is defined via the
multiplication in $V_k$  by $c\circ
v=\bar cv$ for all $c\in\mathbb C$
and $v\in \bar V_k$.
}
\end{equation}

The semilinear mapping
      $\mathcal A_{k-1}$ becomes
      linear in $\mathcal A^{(k)}$.
      Indeed,
if $V_{k-1}\!\!\stackrel{\mathcal
A_{k-1}}{\dashrightarrow}\!\! V_k$,
      then $\mathcal
      A_{k-1}(cu)=\bar c\mathcal
      A_{k-1}u=c\circ \mathcal
      A_{k-1}u$ for all
      $c\in\mathbb C$ and $u\in
      V_{k-1}$. If
$V_{k-1}\!\!\stackrel{\mathcal
A_{k-1}}{\dashleftarrow}\!\! V_k$,
then $\mathcal A_{k-1}(c\circ
      v)=\mathcal A_{k-1}(\bar
      cv)=c\mathcal A_{k-1}v$ for
      all $c\in\mathbb C$ and $v\in
      V_{k}$.

 Similarly, the linear or
      semilinear mapping $\mathcal
      A_{k}$ becomes semilinear or,
      respectively, linear in
      $\mathcal A^{(k)}$.

We repeat procedure \eqref{4KR}
      until obtain a chain in which
      the first $t-1$ mappings are
      linear.

\subsection{A canonical form of the
system of matrices of $\mathcal
A_1,\dots, \mathcal
A_{t-1}$}\label{ss2}

Deleting $\mathcal A_t$ in
\eqref{jsttw} satisfying \eqref{dsr},
we obtain a system of linear mappings
\begin{equation}\label{aau}
{\mathbb{A}} :\quad V_1
\li{{\cal{A}}_1}  V_2
\li{{\cal{A}}_2}
  \cdots
  \li{{\cal{A}}_{t-1}}
  V_t
\end{equation}
in which each line is either
$\longrightarrow$ or $\longleftarrow$.
The classification of such systems is
well known in the theory of quiver
representations (see, for example,
Gabriel's article \cite{gab}, in which
he introduced the notion of quiver
representations and described all
quivers with a finite number of
nonisomorphic indecomposable
representations). We recall this
classification in Lemma \ref{xtu} in a
form that is used in the next section
for reducing the matrix of $\mathcal
A_t$ to canonical form.

A \emph{chain of type $(p,q)$} of
system \eqref{aau}, $1\le p\le q\le t$,
is a sequence of nonzero elements
\[(v_p,v_{p+1},\dots,v_q)\in V_p\times
V_{p+1}\times \dots\times V_q\] such
that
${\cal{A}}_{1},\dots,{\cal{A}}_{t}$ act
on them as follows:
\begin{equation*}\label{dfyj}
0\li{{\cal{A}}_{1}}\cdots
\li{{\cal{A}}_{p-2}}
0\li{{\cal{A}}_{p-1}}
  v_p\li{{\cal{A}}_p}
\cdots\li{{\cal{A}}_{q-1}}
 v_q\li{{\cal{A}}_q}
0\li{{\cal{A}}_{q+1}}\cdots
\li{{\cal{A}}_{t-1}}
 0
\end{equation*}
in which the arrows are directed as in
\eqref{aau}; for simplicity of
notation, we write
$0\xrightarrow{{\cal{A}}_{p-1}} v_p$
instead of
$0\xrightarrow{{\cal{A}}_{p-1}} 0$ and
$v_q\xleftarrow{{\cal{A}}_q} 0$ instead
of $0\xleftarrow{{\cal{A}}_q} 0$.

\begin{lemma}
\label{xtu}
\begin{itemize}
  \item[\rm(a)] For each system of
      linear mappings \eqref{aau},
      we can choose bases
      $E_1,\dots,E_t$ of the spaces
      $V_1,\dots,V_t$ so that
      $E_1\cup\dots\cup E_t$
      consists of disjoint chains.
      The sequence of their types
\begin{equation}\label{kjke}
\underbrace{
(p_1,q_1),\dots,(p_1,q_1)}_{r_1\text{ times}},\dots,
\underbrace{(p_s,q_s),\dots,
(p_s,q_s)}_{r_s\text{ times}},
\end{equation}
in which $(p_i,q_i)\ne(p_j,q_j)$ if
$i\ne j$, is determined by
\eqref{aau} uniquely up to
permutation.

\item[\rm(b)] Suppose that the
    types in \eqref{kjke} are
    numbered so that for each two
    pairs $(p_k,q_k)$ and
    $(p_l,q_l)$
\begin{itemize}

\item[\rm(i)] if $p_k=p_l=1$
    and $q_k<q_l$, then either
    $V_{q_k}\xrightarrow{\mathcal
    A _{q_k}} V_{q_{k+1}}$ and
    $k<l$, or
    $V_{q_k}\xleftarrow{\mathcal
    A _{q_k}} V_{q_{k+1}}$ and
    $k>l$;

\item[\rm(ii)] if $p_k<p_l$ and
    $q_k=q_l=t$, then either
    $V_{p_k-1}\xrightarrow{\mathcal
    A _{p_k-1}} V_{p_k}$ and
    $k>l$, or
    $V_{p_k-1}\xleftarrow{\mathcal
    A _{p_k-1}} V_{p_k}$ and
    $k<l$.
\end{itemize}

Define nonnegative integers
$m_1,\dots,m_s$ and $n_1,\dots,n_s$
by \eqref{kjke} as follows:
\begin{equation}\label{dvu}
m_i:=\begin{cases}
r_i &\text{if }p_i=1,\\
0 &\text{if }p_i>1,\\
\end{cases}\qquad
n_i:=\begin{cases}
r_i &\text{if }q_i=t,\\
0 &\text{if }q_i<t.\\
\end{cases}
\end{equation}

Then the following two conditions
      are equivalent for a pair of
      nonsingular matrices $R$ and
      $S$:
\begin{itemize}
  \item[\rm(i$'$)] There exists
      another system of bases
      $F_1,\dots,F_t$ of the
      spaces $V_1,\dots,V_t$
      that consists of disjoint
      chains with the same
      sequence \eqref{kjke} of
      their types such that $R$
      and $S$ are the change of
      basis matrices from $E_1$
      to $F_1$ in $V_1$ and
      from $E_t$ to $F_t$ in
      $V_t$.

  \item[\rm(ii$'$)] $R$ and $S$
      are upper block
      triangular matrices
\begin{equation}\label{ftn}
R=\begin{bmatrix}
    R_{11}&\dots &R_{1s} \\
&\ddots&\vdots\\
    0 && R_{ss} \\
  \end{bmatrix},\qquad
S=\begin{bmatrix}
    S_{11}&\dots &S_{1s} \\
&\ddots&\vdots\\
    0 && S_{ss} \\
  \end{bmatrix},
\end{equation}
in which every $R_{ii}$ is
$m_i\times n_i$ and every
$S_{ii}$ is $n_i\times n_i$
$($see \eqref{dvu}$)$. If there
is $l$ such that
$(p_l,q_l)=(1,t)$ in
\eqref{kjke}, then
\begin{equation}\label{kus}
R_{ll}=S_{ll}.
\end{equation}
\end{itemize}

\end{itemize}
\end{lemma}

\begin{proof}[Sketch of the proof]
(a) This statement (in another form) is
given in many articles; see, for
example \cite[Section 4]{ser_cycl}.

(b) (i$'$)$\Rightarrow$(ii$'$) Let $R$
and $S$ satisfy (i$'$).

The change of basis matrix $R$ from
$E_1$ to $F_1$ in $V_1$ has the upper
block triangular form \eqref{ftn} by
the following reason. Let
$E_1\cup\dots\cup E_t$ contain two
chains of basis vectors
\begin{align*}
\mathbb  E&:\ \    e_{1}\lin
 \cdots\lin e_{q_k}\li{\mathcal
A_{q_k}} 0\lin
\cdots\lin 0\\
\mathbb  E'&:\ \  e'_{1}\lin
\cdots\lin e'_{q_k}\li{\mathcal
A_{q_k}}
e'_{q_k+1}\lin \cdots\lin
e'_{q_l}\lin 0\lin \cdots\lin
0
\end{align*}
of types $(1,q_k)$ and $(1,q_l)$ with
$q_k<q_l$. We consider two cases
according to the direction of $\mathcal
A_{q_k}$:
\begin{itemize}
  \item If
      $V_{q_k}\xrightarrow{\mathcal
      A_{q_k}} V_{q_k+1}$, then
\[
\mathbb  E'+\alpha
\mathbb  E :\quad  e'_{1}+\alpha
e_{1}\lin \cdots\lin
e'_{q_k}+\alpha
e_{q_k}\xrightarrow{\mathcal
      A_{q_k}} e'_{q_k+1}\lin
\cdots\lin
0
\]
for each $\alpha \in\mathbb C$, and
so we can replace the chain
$\mathbb E'$ by $\mathbb E'+\alpha
\mathbb E$ in $E_1\cup\dots\cup
E_t$. Thus, the basis vector
$e'_{1}\in E_1$ can be replaced by
$e'_{1}+\alpha e_{1}$, and so the
block $R_{kl}$ in \eqref{ftn} is
arbitrary. Due to the condition (i)
in (b), $k<l$.

  \item If
      $V_{q_k}\xleftarrow{\mathcal
      A_{q_k}} V_{q_k+1}$, then
\[
\mathbb  E+\alpha
\tilde{\mathbb  E}' :\quad  e_{1}+\alpha
e'_{1}\lin \cdots\lin e_{q_k}+\alpha
e'_{q_k}\xleftarrow{\mathcal
      A_{q_k}}
0\lin \cdots\lin
0,
\]
in which $\alpha \in\mathbb C$ and
$\tilde{\mathbb
E}'=\{e'_1,\dots,e_{q_k}'\}$ is the
subchain of $\mathbb  E'$. Hence we
can replace the chain $\mathbb E$
by $\mathbb E+\alpha \tilde{\mathbb
E}'$, and so $e_1\in E_1$ can be
replaced by $e_1+\alpha e'_1$.
Therefore, the block $R_{lk}$ in
\eqref{ftn} is arbitrary. Due to
the condition (i) in (b), $k>l$.
\end{itemize}

Analogously, the change of basis matrix
$S$ from $E_t$ to $F_t$ in $V_t$ has
the upper block triangular form
\eqref{ftn} by the following reason.
Let $E_1\cup\dots\cup E_t$ contain two
chains of basis vectors
\begin{align*}
\mathbb  E
 :\  \
0\lin \cdots\lin 0\li{\mathcal
      A_{p_k-1}} e_{p_k}
\lin \cdots\lin e_{t}&\\
\mathbb  E':\ \ 0\lin
\cdots\lin 0\lin
 e'_{p_l}\lin \cdots\lin
e'_{p_k-1}\li{\mathcal
      A_{p_k-1}} e'_{p_k}\lin
\cdots\lin
e'_{t}&
\end{align*}
of types $(p_k,t)$ and $(p_l,t)$ with
$p_l<p_k$.

\begin{itemize}
  \item If
      $V_{p_k-1}\xrightarrow{\mathcal
      A_{p_k-1}} V_{p_k}$, then
\[
\mathbb  E+\alpha
\tilde{\mathbb  E}' :\ \
0\lin \cdots\lin 0\xrightarrow{\mathcal
      A_{p_k-1}}
e_{p_k}+\alpha e'_{p_k}\lin
\cdots\lin e_{t}+\alpha
e'_{t},
\]
in which $\alpha \in\mathbb C$ and
$\tilde{\mathbb  E}'=\{e'_{p_k},
\dots, e'_{t}\}$ is the subchain of
$\mathbb  E'$. Hence, we can
replace the chain $\mathbb E$ by
$\mathbb E+\alpha \tilde{\mathbb
E}'$, and so the basis vector
$e_{t}\in E_t$ can be replaced by
$e_{t}+\alpha e'_{t}$. Therefore,
the block $R_{lk}$ in \eqref{ftn}
is arbitrary. Due to the condition
(ii) in (b), $k>l$.

  \item If
      $V_{p_k-1}\xleftarrow{\mathcal
      A_{p_k-1}} V_{p_k}$, then
\[
\mathbb  E'+\alpha
\mathbb  E:\ \
0\lin \cdots\lin
e'_{p_k-1}\xleftarrow{\mathcal
      A_{p_k-1}}
e'_{p_k}+\alpha e_{p_k}\lin
\cdots\lin
e'_{t}+\alpha e_{t}
\]
for each $\alpha \in\mathbb C$, and
so we can replace the chain
$\mathbb E'$ by $\mathbb E'+\alpha
\mathbb  E$. Thus, the basis vector
$e'_{t}\in E_t$ can be replaced by
$e'_{t}+\alpha e_{t}$. Therefore,
the block $R_{kl}$ in \eqref{ftn}
is arbitrary. Due to the condition
(ii) in (b), $k<l$.
\end{itemize}

Let $l$ be such that $(p_l,q_l)=(1,t)$
in \eqref{kjke}, let $r_l\ge 2$, and
let
\begin{align*}
\mathbb  E&:\quad  e_{1}\lin
 e_{2}\lin
\:\cdots\:\lin e_{t}\\
\mathbb  E' &:\quad  e'_{1}\lin
 e'_{2}\lin
\cdots\lin e'_{t}
\end{align*}
be two chains. Then
\[
\mathbb  E+\alpha
\mathbb  E':\ \  e_{1}
+\alpha e'_{1}\lin
 e_{2}+\alpha e'_{2}\lin
\:\cdots\:\lin e_{t}+\alpha e'_{t}
\]
for each $\alpha \in\mathbb C$. Hence
the basis vectors $e_{1}\in E_1$ and
$e_{t}\in E_t$ can be simultaneously
replaced by $e_{1}+\alpha e'_{1}$ and
$e_{t}+\alpha e'_{t}$. Therefore,
$R_{ll}=S_{ll}$ in \eqref{ftn}.

We have shown that $R$ and $S$ satisfy
(ii$'$).

(i$'$)$\Leftarrow$(ii$'$) This
implication follows from the above
reasoning.
\end{proof}

\subsection{A canonical form of the
matrix of $\mathcal A_t$}\label{ss3}

Let $\cal A$ be a cycle of linear and
semilinear mappings \eqref{jsttw}
satisfying the condition \eqref{dsr}.
We suppose that $V_1\!\stackrel{{\cal
A} _t}{\longrightarrow}\! V_t$ or
$V_1\!\!\!\stackrel{{\cal A}
_t}{\dashrightarrow}\!\!\! V_t$
 (if $V_1\stackrel{{\cal A}
_t}{\longleftarrow} V_t$ or
$V_1\!\!\stackrel{{\cal A}
_t}{\dashleftarrow}\!\! V_t$, then we
renumber the vector spaces
$V_1,\dots,V_t$ in the reverse order).

Deleting $\mathcal A_t$ from
\eqref{jsttw}, we obtain the chain of
linear mappings \eqref{aau}. By Lemma
\ref{xtu}, there exists a system of
      bases $E_1,\dots,E_t$ of the
      spaces $V_1,\dots,V_t$ that
      consists of disjoint chains.
Moreover, the bases $E_1$ and $E_t$ are
determined up to replacement by bases
$F_1$ and $F_t$ such that the change of
basis matrices $R$ and $S$ have the
form \eqref{ftn}. Thus, the matrix
$A_t$ of the linear or semilinear
mapping $\mathcal A_t$ is reduced by
transformations
\begin{equation}\label{lot}
A_t\mapsto S^{-1}A_tR\text{ or }
\bar S^{-1}A_tR,
\qquad \text
{$R$ and $S$ have the form \eqref{ftn}.}
\end{equation}
This leads to the following definition.

\begin{definition}\label{zzz}
By a \emph{marked block matrix} we mean
a block matrix in which some of square
blocks are crossed along their main
diagonal by  a full or dashed line and
each horizontal or vertical strip
contains at most one crossed block. A
\emph{block matrix problem} is the
canonical form problem for marked block
matrices with respect to the following
admissible transformations:
\begin{itemize}
  \item[(i)] arbitrary elementary
      transformations within
      vertical and horizontal
      strips such that each crossed
      block is transformed by
similarity transformations if its
      cross-line is full and by
      consimilarity transformations
      if the cross-line is dashed;

  \item[(ii)] additions of a row
      multiplied by scalar to a row
      of a horizontal strip that is
      located above, and of a
      column multiplied by scalar
      to a column of a vertical
      strip that is located to the
      right.
\end{itemize}
We say that $M$ and $N$ are
\emph{equivalent} if $M$ is reduced to
$N$ by transformations (i) and (ii).
The problem is to find a canonical form
of a marked block matrix up to
equivalence.
\end{definition}

We use the methods of \cite{n-r-s-b},
in which the block matrix problem was
solved for marked block matrices
without blocks that are crossed by
dashed lines.

It follows from the form of $R$ and $S$
in \eqref{ftn} that the canonical form
problem for $A_t$ with respect to
transformations \eqref{lot} is the
block matrix problem for $A_t$
partitioned into horizontal strips
conformally to the partition of $S$ and
into vertical strips conformally to the
partition of $R$; the condition
\eqref{kus} means that the $(l,l)$
block of $A_t$ is crossed by a full or
dash line. (Note that $A_t$ may contain
horizontal strips with no rows and
vertical strips with no columns, which
follows from \eqref{dvu}.)

Let $M=[M_{ij}]$ and $N=[N_{ij}]$ be
two block matrices with the same number
of horizontal strips, the same number
of vertical strips, the same
disposition of blocks crossed by full
lines, and the same disposition of
blocks crossed by dashed lines. The
\emph{block direct sum} of $M$ and $N$
is the block matrix
\begin{equation*}\label{few}
 M\boxplus N:=[M_{ij}\oplus N_{ij}]
\end{equation*}
with the same disposition of blocks
crossed by full lines and the same
disposition of blocks crossed by dashed
lines. We say that $M$ is
\emph{indecomposable} if it is not
$0$-by-$0$ and it is not equivalent to
a block direct sum of matrices of
smaller sizes.

We say that a matrix is \emph{empty} if
it does not contain rows or columns.

\begin{lemma}\label{kwq}
Let a marked block matrix $M$ be
indecomposable. Then either
\begin{itemize}
  \item[\rm(a)] all blocks of $M$
      are empty except for one
      crossed block that is
      nonsingular, or
  \item[\rm(b)] $M$ is equivalent
      to a matrix whose entries are
      only 0's and 1's with at most
      one 1 in each row and each
      column.
\end{itemize}
\end{lemma}

\begin{proof}
Let $M\ne 0$. We use induction on the
size of $M$. Let $M_{pq}$ be the lowest
nonzero block in the first nonzero
vertical strip.

\emph{Case 1: $M_{pq}$ is not crossed.}
We reduce it by transformations (i) to
the form
\begin{equation}\label{cly}
M_{pq}=\begin{bmatrix}
         0&I_k \\
         0 & 0 \\
       \end{bmatrix},\qquad k\ge 1,
\end{equation}
and then use only those transformations
(i) and (ii) that preserve it.

The matrix \eqref{cly} is partitioned
into 2 horizontal and 2 vertical
strips. We extend this partition to the
entire $p$th horizontal strip and to
the entire $q$th vertical strip of $M$.
Using transformations (ii), we make
zero all entries over $I_k$ and to the
right of $I_k$ in the new strips.

Let the new division lines do not pass
through crossed blocks. Then $I_k$ is a
direct summand of $M$. Since $M$ is
indecomposable, we have that $k=1$,
$M_{pq}=I_1$, and all other blocks are
empty, which proves the lemma in this
case.

Let a new division line pass through a
crossed block $M_{lr}$ ($l=p$ or
$r=q$). Then we draw the perpendicular
division line such that $M_{lr}$ is
partitioned into 4 subblocks with
square diagonal blocks (they are
crossed by the full or dash line that
crosses $M_{lr}$):
\begin{equation}\label{ksr}
 M_{lr}=\begin{bmatrix}
          A & B \\
          C & D \\
        \end{bmatrix},
\qquad \begin{matrix}
\text{$A$ and $D$ are crossed,} \\
\text{$A$ is $k\times k$ if $l=p$,
$D$ is $k\times k$ if $r=q$.}
\end{matrix}
\end{equation}
If a new division line passes through
another crossed block, we repeat this
procedure.

Denote by $M'$ the marked block matrix
obtained from $M$ by deleting the new
vertical and horizontal strips that
pass through $I_k$ in $M_{pq}$ (if $M$
is $m$-by-$n$, then $M'$ is
$(m-k)$-by-$(n-k)$). It is easy to see
that all transformations (i) and (ii)
with $M'$ can be obtained from
transformations (i) and (ii) with $M$.
Reasoning by induction, we assume that
the lemma holds for $M'$.

Note that $M'$ cannot satisfy the
condition (a) since $M'$ contains a
strip with $k$ rows or columns and
without crossed blocks; this strip is
obtained from the strip of $M$ that
contains $A$ in \eqref{ksr} if $l=p$ or
$D$ if $r=q$. Therefore, $M'$ satisfies
the condition (b), and so $M$ satisfies
(b) too.

\emph{Case 2: $M_{pq}$ is crossed.} The
Jordan canonical form for similarity
(if the cross-line is full) and Theorem
\ref{tej} (if the cross-line is dashed)
ensure that $M_{pq}$ can be reduced to
the form $N\oplus J$, in which $N$ is
nonsingular and $J$ is a nilpotent
Jordan matrix.

If $N$ is nonempty, then we make zero
all entries over $N$ and to the right
of $N$ by transformations (ii). Since
$M$ is indecomposable, $M_{pq}=N$ and
the other blocks of $M$ are empty, and
so $M$ satisfies the condition (a).
Thus, we can suppose that $M=J$.

Collecting the Jordan blocks of the
same size by permutations of rows and
the same permutations of columns, we
reduce $M_{pq}$ to the form
\begin{equation}\label{dwi}
M_{pq}=
\begin{MAT}(b){cccccccc}
0_u & 0&0 & 0&0&0&\cdots&\scriptstyle \it 1\\
0 & 0&I_v & 0&0&0&\cdots&\scriptstyle \it 2\\
0 & 0&0 & 0&0&0&\cdots&\scriptstyle \it 3\\
0 & 0&0 & 0&I_w&0&\cdots&\scriptstyle \it 4\\
0 & 0&0 & 0&0&I_w&\cdots&\scriptstyle \it 5\\
0 & 0&0 & 0&0&0&\cdots&\scriptstyle \it 6\\
\vdots &\vdots&\vdots &\vdots&
\vdots&\vdots&\ddots&\\
\scriptstyle \it 1&\scriptstyle \it 2
&\scriptstyle \it 3&\scriptstyle \it 4
&\scriptstyle \it 5&\scriptstyle \it 6&&
\addpath{(1,8,4)dddrrrrrdddllluuuuulll}
\addpath{(3,7,.)rrrr}
\addpath{(0,5,.)r}
\addpath{(6,5,.)r}
\addpath{(0,2,.)rrr}
\addpath{(6,2,.)r}
\addpath{(1,1,.)uuuu}
\addpath{(3,1,.)u}
\addpath{(3,7,.)u}
\addpath{(6,8,.)ddd}
\addpath{(6,2,.)d}
\addpath{(0,1,4)uuuuuuurrrrrrrdddddddlllllll}
\addpath{(0,3,.)rrrrrrr}
\addpath{(0,4,.)rrrrrrr}
\addpath{(0,6,.)rrrrrrr}
\addpath{(2,1,.)uuuuuuu}
\addpath{(4,1,.)uuuuuuu}
\addpath{(5,1,.)uuuuuuu}
\\
\end{MAT}\,,
\qquad u,v,w,\dots\ge 0.
\end{equation}
We extend the obtained partition of
$M_{pq}$ to the entire $p$th horizontal
strip and the entire $q$th vertical
strip of $M$ and make zero all entries
over and to the right of
$I_v,I_w,I_w,\dots$.

Denote by $M'$ the submatrix of $M$
obtained from $M$ by deleting the
``new'' horizontal and vertical strips
containing $I_v,I_w,I_w,\dots$. In
particular, $M_{pq}$ converts to
\begin{equation*}\label{dahr}
M'_{pq}=
\begin{MAT}(b){ccccc}
0_u & 0&0 & \cdots&\scriptstyle \it 1\\
0 & 0_v & 0&\cdots&\scriptstyle \it 3\\
0 & 0& 0_w & \cdots&\scriptstyle \it 6\\
\vdots &\vdots &\vdots&\ddots&
\\
\scriptstyle \it 1&\scriptstyle \it 2
&\scriptstyle \it 4&&
\addpath{(1,5,4)ddrrdluull}
\addpath{(0,3,.)rdd}
\addpath{(0,2,.)rrd}
\addpath{(2,5,.)drr}
\addpath{(3,5,.)ddr}
\addpath{(3,1,.)ur}
\addpath{(0,1,4)uuuurrrrddddllll}
\\
\end{MAT}
\end{equation*}
The set of crossed blocks of $M'$
consists of the crossed blocks of $M$
(except for $M_{pq}$) and the diagonal
blocks $0_u,0_v,0_w,\dots$ of $M_{pq}'$
If $M_{pq}$ is crossed by a full line,
then $0_u,0_v,0_w,\dots$ are crossed by
full lines too. If $M_{pq}$ is crossed
by a dash line, then all odd-numbered
diagonal blocks $0_u,0_w,\dots$ of
$M_{pq}$ (they are obtained from the
Jordan blocks of $M_{pq}$ of odd sizes)
are crossed by dash lines; the
even-numbered diagonal blocks
$0_v,\dots$ of $M_{pq}$ are crossed by
full lines.

Denote by $M''$ the marked block matrix
obtained from $M'$ by arrangement of
its ``new'' vertical strips (passing
through $M_{pq}'$) in reverse order. In
particular, $M'_{pq}$ converts to
\begin{equation*}\label{dhr}
M''_{pq}=
\begin{MAT}(b){ccccc}
\cdots& 0 & 0&0_u &\scriptstyle \it 1\\
\cdots&0 & 0_v & 0&\scriptstyle \it 3\\
\cdots&0_w & 0& 0 &\scriptstyle \it 6\\
\udots&\vdots &\vdots &\vdots&
\\
&\scriptstyle \it 4
&\scriptstyle \it 2&\scriptstyle \it 1&
\addpath{(3,5,4)ddlldruurr}
\addpath{(4,3,.)ldd}
\addpath{(4,2,.)lld}
\addpath{(2,5,.)dll}
\addpath{(1,5,.)ddl}
\addpath{(1,1,.)ul}
\addpath{(0,1,4)uuuurrrrddddllll}
\\
\end{MAT}
\end{equation*}

Let us prove that
\begin{equation}\label{4.26}
\parbox{25em}
{if $M$ is reduced by transformations
(i) and (ii) that preserve $M_{pq}$
(which is of the form \eqref{dwi}), then $M''$ is reduced
by transformations
(i) and (ii).}
\end{equation}

We first prove \eqref{4.26} for
transformations (ii). Preserving
$M_{pq}$, we can make the following
transformations with strips of $M$ that
 $M_{pq}$:
\begin{itemize}
  \item Add columns of the $2$nd
      vertical strip to columns of
      the $1$st vertical strip.
      Since $M_{pq}$ is transformed
      by (con)similarity
      transformations, we must do
      the (con)inverse
      transformations with rows of
      $M_{pq}$: to subtract the
      corresponding columns (or the
      complex conjugate columns) of
      the $1$st horizontal strip
      from the $2$nd horizontal
      strip, which does not change
      $M_{pq}$ since its $1$st
      horizontal strip is zero.
      Thus, we can add in
      $M_{pq}''$ columns passing
      through $0_v$ to columns
      passing through $0_u$.

  \item Add columns of the $4$th
      vertical strip to columns of
      the $1$st vertical strip; the
      (con)inverse transformations
      does not change $M_{pq}$.
      Thus, we can add in
      $M_{pq}''$ columns passing
      through $0_w$ to columns
      passing through $0_u$.

  \item Add columns of the $4$th
      vertical strip to columns of
      the $2$nd vertical strip. The
      (con)inverse transformations
      with rows of $M_{pq}$ spoil
      the (4,3) block. It is
      restored by additions of
      columns of the $5$th vertical
      strip; the (con)inverse
      transformations do not change
      $M_{pq}$. Thus, we can add in
      $M_{pq}''$ columns passing
      through $0_w$ to columns
      passing through $0_v$.
\\
      \dots

  \item Add rows of the $3$rd and
      $6$th horizontal strips to
      rows of the $1$st horizontal
      strip; the (con)inverse
      transformations with columns
      do not change $M_{pq}$. Thus,
      we can add in $M_{pq}''$ rows
      passing through $0_v$ and
      $0_w$ to rows passing through
      $0_u$.

  \item Add rows of the $6$th
      horizontal strip to rows of
      the $3$rd horizontal strip.
      The (con)inverse
      transformations with columns
      of $M_{pq}$ spoil the (2,6)
      block. It is restored by
      additions of rows of the
      $5$th horizontal strip; the
      (con)inverse transformations
      do not change $M_{pq}$. Thus,
      we can add in $M_{pq}''$ rows
      passing through $0_w$ to rows
      passing through $0_v$.\\
      \dots
\end{itemize}

Therefore, preserving $M_{pq}$ we can
make transformations (ii) with $M''$,
which proves \eqref{4.26} for
transformations (ii).

Let us prove \eqref{4.26} for
transformations (i).

If $M_{pq}$ is crossed by a full line
(i.e., $M_{pq}$ is transformed by
similarity transformations), then we
take $S:=S_1\oplus S_2\oplus S_2\oplus
S_3 \oplus S_3\oplus S_3\oplus \cdots $
(in which $S_1$ is $u\times u$, $S_2$
is $v\times v$,\dots) and obtain $
S^{-1}M_{pq}S=M_{pq}$. Thus, $M'_{pq}$
can be reduced by transformations
\[
(S_1\oplus S_2\oplus
S_3\oplus\cdots)^{-1}M'_{pq} (S_1\oplus
S_2\oplus S_3\oplus\cdots),\] and so
$M''_{pq}$ can be reduced by
transformations (i).

Assume now that $M_{pq}$ is crossed by
a dashed line; that is, $M_{pq}$ is
transformed by consimilarity
transformations. Then
\[
\bar S^{-1}M_{pq}S=M_{pq}\qquad\text{if }
 S:=S_1\oplus S_2\oplus \bar S_2\oplus S_3
\oplus \bar S_3\oplus S_3\oplus \cdots
\]
which is illustrated as follows:
\[
\begin{MAT}(b){cccccccc}
&\scriptstyle S_1&\scriptstyle S_2&
\scriptstyle \bar S_2&\scriptstyle S_3
&\scriptstyle \bar S_3&\scriptstyle S_3&\\
\scriptstyle \bar S_1^{-1}&0_u & 0&0 & 0&0&0&\cdots\\
\scriptstyle \bar S_2^{-1}&0 & 0&I_v & 0&0&0&\cdots\\
\scriptstyle S_2^{-1}&0 & 0_v&0 & 0&0&0&\cdots\\
\scriptstyle \bar S_3^{-1}&0 & 0&0 & 0&I_w&0&\cdots\\
\scriptstyle S_3^{-1}&0 & 0&0 & 0&0&I_w&\cdots\\
\scriptstyle \bar S_3^{-1}&0 & 0&0 & 0_w&0&0&\cdots\\
&\vdots &\vdots&\vdots &\vdots&
\vdots&\vdots&\ddots
\addpath{(2,7,3)dddrrrrrdddllluuuuulll}
\addpath{(2,5,4)drul}
\addpath{(1,7,4)drul}
\addpath{(4,2,4)drul}
\addpath{(4,6,.)rrrr}
\addpath{(1,4,.)r} \addpath{(7,4,.)r}
\addpath{(1,1,.)rrr} \addpath{(7,1,.)r}
\addpath{(2,0,.)uuuu}
\addpath{(4,0,.)u} \addpath{(4,6,.)u}
\addpath{(7,7,.)ddd} \addpath{(7,1,.)d}
\addpath{(1,0,4)uuuuuuurrrrrrrdddddddlllllll}
\addpath{(1,2,.)rrrrrrr}
\addpath{(1,3,.)rrrrrrr}
\addpath{(1,5,.)rrrrrrr}
\addpath{(3,0,.)uuuuuuu}
\addpath{(5,0,.)uuuuuuu}
\addpath{(6,0,.)uuuuuuu}
\\
\end{MAT}
\]
Thus, $0_u,0_w,\dots$ are transformed
by consimilarity transformations and
$0_v,\dots$ by similarity
transformations. This proves
\eqref{4.26}.

Reasoning by induction, we assume that
the lemma holds for $M''$. The matrix
$M''$ cannot satisfy the condition (a)
of the lemma since $M_{pq}''$ is
nonempty. Hence, $M''$ satisfies (b),
then $M''$ satisfies (b) too.
\end{proof}

\subsection{Completion of the proof
of Theorem \ref{t1.1}}\label{ss4}

Let $\cal A$ be a system of linear and
semilinear mappings \eqref{jsttw} whose
sequence of matrices of
${\cal{A}}_{1},\dots,{\cal{A}}_{t}$
cannot be decomposed into a direct sum
of systems of matrices of smaller
sizes. By Lemma \ref{xtu}, there exists
a system of
      bases $E_1,\dots,E_t$ of the
      spaces $V_1,\dots,V_t$ that
      falls into disjoint chains.
By Lemma \ref{kwq}, the bases
$E_1,\dots,E_t$ can be chosen such that
either
\begin{itemize}
  \item[(a)] the matrices $
      A_1,\dots,A_{t-1}$ of
      $\mathcal A_1,\dots,\mathcal
      A_{t-1}$ are the identity and
      the matrix $A_t$ of $\mathcal
      A_{t}$ is nonsingular, or
  \item[(b)] if ${\cal A}_i:V_k\to
      V_l$ or ${\cal
      A}_i:V_k\dashrightarrow V_l$
      and $e\in E_k$, then ${\cal
      A}_ie\in E_l$ or ${\cal
      A}_ie=0$; moreover, if
      $e,f\in E_k$ and ${\cal
      A}_ie={\cal A}_if\ne 0$, then
      $e=f$.
\end{itemize}

Consider the case (a). Let
$S_1,\dots,S_t$ be the change of basis
matrices that preserve
$A_1=\dots=A_{t-1}=I$, then
$S_i=S_{i+1}$ if the arrow between $i$
and $i+1$ is full and $S_i=\bar
S_{i+1}$ if the arrow is dashed.
Therefore, $A_t$ is reduced by
similarity transformations
$S^{-1}_1A_tS_1$ or consimilarity
transformations $\bar S^{-1}_1A_tS_1$
if the number of dashed arrows in
\eqref{jsttw} is even or odd,
respectively. Using the Jordan
canonical form or Theorem \ref{tej}, we
obtain a sequence of matrices (i) from
Theorem \ref{t1.1}.

Consider the case (b). Let us construct
the directed graph, whose set of
vertices is the set of basis vectors
$E_1\cup\dots\cup E_t$ and there is an
arrow from $u$ to $v$ if and only if
${\cal}A_iu=v$ for some $i=1,\dots,t$.
In the case (b), this graph is a
disjoint union of chains. Since the
sequence of matrices of
${\cal{A}}_{1},\dots,{\cal{A}}_{t}$
cannot be decomposed into a direct sum,
the graph is connected, and so it is a
chain.

For example, a system
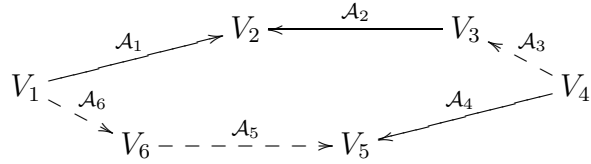
\[
\xymatrix@R=6pt{
&&V_2\ar@{<-}[rr]^{\mathcal A_2}&&
V_3\ar@{<--}[dr]^{\mathcal A_3}&\\
V_1\ar@{->}[urr]^{\mathcal A_1}&&&&&
V_4\ar@{->}[dll]_{\mathcal A_4}\\
&V_6\ar@{<--}[ul]_{\mathcal A_6}&&
V_5\ar@{<--}[ll]_{\mathcal A_5}&&\\}
\]
may have the chain
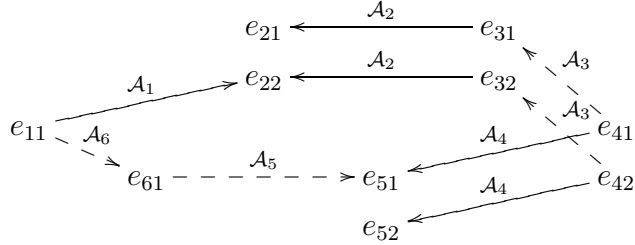
\[
\xymatrix@R=6pt{
&&e_{21}\ar@{<-}[rr]^{\mathcal A_2}&&e_{31}
\ar@{<--}[ddr]^{\mathcal A_3}&\\
&&e_{22}
\ar@{<-}[rr]^{\mathcal A_2}&&e_{32}
\ar@{<--}[ddr]^{\mathcal A_3}&\\
e_{11}\ar@{->}[urr]^{\mathcal A_1}&&&&&
e_{41}\ar@{->}[dll]_{\mathcal A_4}\\
&e_{61}\ar@{<--}[ul]_{\mathcal A_6}&&
e_{51}\ar@{<--}[ll]_{\mathcal A_5}&&e_{42}
\ar@{->}[dll]_{\mathcal A_4}\\
&&&e_{52}\\}
\]
Its mappings ${\cal A}_1,\dots,{\cal
A}_6$ are given by the matrices
\[
A_1=\begin{bmatrix}
   0\\1
    \end{bmatrix},\quad
A_2=A_3=A_4=\begin{bmatrix}
   1&0\\0&1
    \end{bmatrix},\quad
A_5=\begin{bmatrix}
   1\\0
    \end{bmatrix},\quad
A_6=\begin{bmatrix}
   1
    \end{bmatrix}.
\]
They form a sequence (iii) from Theorem
\ref{t1.1} with $(A_i,A_j)=(A_1,A_5)=
(F_2^T,G_2^T) $. It is easy to see that
all matrix sequences (ii) and (iii)
from Theorem \ref{t1.1} can be obtained
analogously.

We have proved that for each cycle of
linear and semilinear mappings
\eqref{jsttw} there exist bases of the
spaces $V_1,\dots,V_t$ in which the
sequence of matrices of
${\cal{A}}_{1},\dots,{\cal{A}}_{t}$ is
a direct sum of sequences of the form
(i)--(iii); the uniqueness of this
direct sum follows from the
Krull--Schmidt theorem for additive
categories \cite[Chapter I, Theorem
3.6]{bas} (it holds for cycles
\eqref{jsttw} since they form an
additive category in which all
idempotents split).


\begin{thebibliography}{99}


\bibitem{bas} H. Bass, Algebraic
 $K$-theory, Benjamin, New York,
 1968.


\bibitem{djok} D.\v{Z}. Djokovi\'{c},
    Classification of pairs consisting
    of a linear and a semilinear map,
    Linear Algebra Appl. 20 (1978)
    147--165.

\bibitem{dob+pon} N.M.
    Dobrovol$'$skaya, V.A.
    Ponomarev, A pair of
    counter-operators (in Russian),
    Uspehi Mat. Nauk 20 (no. 6)
    (1965) 80--86; MR 36\#2631.

\bibitem{debora} D. Duarte de Oliveira,
    R.A. Horn, T. Klimchuk, V.V.
    Sergeichuk, Remarks on the
    classification of a pair of
    commuting semilinear operators,
Linear
    Algebra Appl. 436 (2012)
    3362--3372.


\bibitem{gab} P. Gabriel, Unzerlegbare
    Darstellungen I, {Manuscripta
    Math.} 6 (1972) 71--103.

\bibitem{gab_roi} P. Gabriel, A.V.
    Roiter,
    Representations of
    Finite-Dimensional Algebras,
Springer-Verlag,
    1997.

\bibitem{haa} J. Haantjes,
    Klassifikation der antilinearen
    Transformationen, Math. Ann. 112
    (1935) 98--106.

\bibitem{hon-hor} Y.P. Hong, R.A. Horn,
    A canonical form for matrices under
    consimilarity, Linear Algebra Appl.
    102 (1988) 143--168.

\bibitem{HJ12} R.A. Horn, C.R. Johnson,
    Matrix Analysis, 2nd ed., Cambridge
    University Press, New York, 2012.

\bibitem{hor+mer} R.A. Horn, D.I.
    Merino, Contragredient equivalence:
    a canonical form and some
    applications, Linear Algebra
    Appl.  214 (1995) 43--92.

\bibitem{jordan} C. Jordan,
    Trait\'{e} des
    Substitutions et des \'Equations
Alg\'ebriques, Gauthier-Villars, Paris,
    1870; available online at
{\footnotesize
\verb"http://visualiseur.bnf.fr/CadresFenetre?O=NUMM-29053&I=143&M=tdm"
}

\bibitem{kro} L. Kronecker,
    Algebraische Reduktion der Scharen
    bilinearer Formen,
    Sitzungsber. Akademie Berlin
    (1890) 763--776.

\bibitem{n-r-s-b}
    L.A.
    Nazarova,
    A.V.
    Roiter, V.V. Sergeichuk, V.M.
    Bondarenko, Application of modules
    over a dyad for the classification
    of finite $p$-groups possessing an
    abelian subgroup of index $p$ and
    of pairs of mutually annihilating
    operators, J. Soviet Math. 3
    (no. 5) (1975) 636--654.


\bibitem{ser_cycl} V.V. Sergeichuk,
    Computation  of  canonical
    matrices for  chains  and  cycles
    of  linear mappings, Linear Algebra
    Appl. 376 (2004) 235--263.

\end{thebibliography}
\end{document}